\pdfoutput=1

\documentclass[11pt]{amsart}
\usepackage{Alegreya}

\newtheorem{theorem}{Theorem}[subsection]
\newtheorem{theorem*}{Theorem}
\newtheorem{lemma}[theorem]{Lemma}
\newtheorem{corollary}[theorem]{Corollary}
\newtheorem{definition}[theorem]{Definition}
\newtheorem{example}[theorem]{Example}

\theoremstyle{remark}
\newtheorem{remark}[theorem]{Remark}
\newtheorem{conjecture}[theorem]{Conjecture}
\newtheorem{proposition}[theorem]{Proposition}

\numberwithin{equation}{subsection}
\usepackage{enumerate}
\usepackage{tikz-cd}
\usetikzlibrary{bending}
\usepackage{graphicx,adjustbox}
\usepackage{hyperref}
\usepackage{paralist}
\usepackage{amsmath,amssymb}
\usepackage[normalem]{ulem}
\usepackage{upgreek}
\hypersetup{colorlinks,linkcolor={blue},citecolor={red},urlcolor={red}}
\usepackage[a4paper,
            bindingoffset=0.2in,
            left=1in,
            right=1in,
            top=1in,
            bottom=1in,
            footskip=.25in]{geometry}

\usepackage[
backend=biber,
bibstyle=numeric,
sorting=nyt,
]{biblatex}
\DeclareNameAlias{author}{last-first}
\DeclareFieldFormat[article, inbook]{title}{#1} 
\renewbibmacro{in:}{}
\DeclareFieldFormat{journaltitle}{#1\isdot}

\usepackage{amscd}
\usepackage{graphicx}
\usepackage[all]{xy}
\title[Very stable and wobbly loci for elliptic curves]{Very stable and wobbly loci for elliptic curves}
\author{Kuntal Banerjee}
\address{Centre for Quantum Topology and Its Applications (quanTA) and Department of Mathematics and Statistics, University of Saskatchewan, SK, Canada~ S7N 5E6}
\email{kub129@usask.ca}
\author{Steven Rayan}
\address{Centre for Quantum Topology and Its Applications (quanTA) and Department of Mathematics and Statistics, University of Saskatchewan, SK, Canada~ S7N 5E6}
\email{rayan@math.usask.ca}

\begin{document}
\begin{abstract}
    We explore very stable and wobbly bundles, twisted in a particular sense by a line bundle, over complex algebraic curves of genus $1$. We verify that twisted stable bundles on an elliptic curve are not very stable for any positive twist. We utilize semistability of trivially twisted very stable bundles to prove that the wobbly locus is always a divisor in the moduli space of semistable bundles on a genus $1$ curve. We prove, by extension, a conjecture regarding the closedness and dimension of the wobbly locus in this setting. This conjecture was originally formulated by Drinfeld in higher genus.
\end{abstract}
\maketitle 
\tableofcontents

\let\thefootnote\relax\footnotetext{2020\textit{ Mathematics Subject Classification.} 14C20, 14H52.}\let\thefootnote\relax

\section{Statement of results}\label{state}

Laumon introduced \emph{very stable bundles}, proving that the very stable locus is an open dense subset in the moduli space of stable bundles \cite{Gera}. Hausel subsequently studied \cite{Hausel} the global nilpotent cone for $\mbox{SL}(2,\mathbb C)$-Higgs bundles with a fixed determinant line bundle $\Lambda$ of degree $1$ on a curve of genus $\geq 2$. The zero set of the determinant morphism (valued in global sections of $K_X^2$) contains the moduli of stable bundles of rank $2$ with fixed determinant $\Lambda$, as an open subset. The inclusion is simply $[E]\mapsto [(E, 0)]$. This nilpotent cone coincides with the downward Morse flow of the moduli space of Higgs bundles (cf. \cite{Hausel}, Theorem 4.4.2). The \emph{wobbly locus} $\mathcal{W}$ is the collection of stable bundles that admit a nonzero nilpotent Higgs field. This definition of the wobbly locus was exploited by Pal and Pauly \cite{Palpau}, who investigated the irreducible components of the wobbly locus in order to prove a conjecture that Laumon attributes to Drinfeld in \cite{Gera}. The claim is that the wobbly locus $\mathcal{W}$ is a closed subset of pure codimension $1$. In a separate paper \cite{PalG}, Pal proves Drinfeld's conjecture in arbitrary rank. However, the irreducible components are not easily understood for higher ranks because we lose access to useful information about their dimensions. In spite of this, Pauly and Pe\'on-Nieto realized the importance of locating a very stable bundle $E$ via the sections of its twisted endomorphism bundle $\mbox{End}E\otimes K_X$ (cf. \cite{Peon}), imposing a condition of properness and appealing to the semiprojectivity of the moduli space. More recently, Hausel and Hitchin \cite{Hausel2021VerySH} introduced \emph{very stable Higgs bundles} by appealing to the upward flow of $\mathbb C^\times$-action on the moduli space of Higgs bundles.\\

Working in this context, we focus on curves of genus $0$ and $1$. It is common to investigate moduli spaces of parabolic Higgs bundles on such curves; however, instead of puncturing the curve at certain marked points and maintaining the canonical line bundle as the sheaf of values of the Higgs field, we keep the curve compact but allow the Higgs field to take values in \emph{any} line bundle $L$.  We say that the Higgs field is \emph{twisted} by $L$ or, accordingly, that $L$ is the \emph{twist} of the Higgs field. Thus far, we have not encountered literature concerning very stable and wobbly loci in this twisted setting. In order to characterize very stable and wobbly bundles here, we shall need to rely on several well-known results about the geometry of symmetric products of a smooth projective curve. For sufficiently large $n$, the $n$-fold symmetric product of a curve $X$, $\mbox{Sym}^n(X)$, can be realized as a projective bundle over $\mbox{Pic}^0(X)$. Carefully studying $\mbox{Sym}^n(X)$ for arbitrary curves $X$ and any $n\geq1$, Macdonald \cite{Mac} computed their Betti numbers and the cohomological invariants of their closed subvarieties. Hitchin \cite{Hitch2} and Gothen \cite{gothen} utilized the now classical Macdonald's formula for the Betti numbers of $\mbox{Sym}^n(X)$ to compute Poincar\'e polynomials of moduli spaces of stable Higgs bundles for ranks up to and including $3$ on curves of genus $g\geq2$. In a context close the spirit of the present manuscript, a similar application of thi formula was used to compute Betti numbers of moduli spaces of $A$-type quiver bundles in low genus, as a means to accessing similar results for twisted Higgs bundle moduli spaces at genus $0$ and $1$ \cite{Ste2}.  We refer the reader to \cite{aspects} for more on these subjects.  In our methods, we will rely on Macdonald's formula in at least one major step.  Our techniques also depend intrinsically upon results about divisors in symmetric products as per, for example, \cite{Kou,Mus,Arb}.

Now, we will establish some of the basic notation required to state our results. Let us continue to use $X$ to denote a smooth complex algebraic curve; likewise, $L$ continues to denote a choice of holomorphic line bundle on $X$. Let $E\to X$ be a holomorphic vector bundle of rank $r$ so that the only nilpotent element $\phi\in H^0(\text{End}E\otimes L)$ is $0$. Such a bundle is said to be an $L$-{\textit{very stable}} bundle. We call a bundle $L$-{\textit{wobbly}} if it is not $L$-very stable. We will refer the $L$-very stable (respectively, the $L$-wobbly) bundles to \textit{very stable} (respectively, \textit{wobbly}) if the twist line bundle $L$ is well understood from the context. In particular, we call a bundle \textit{canonically very stable} (respectively, \textit{canonically wobbly}) if $L = K_X$ is the canonical line bundle on $X$. In the literature of which we are aware, the canonically very stable bundles are referred as ``very stable bundles''. We remark that any line bundle is immediately very stable, while a general bundle $E$ is $L$-very stable if and only if $E\otimes M$ is very stable for a line bundle $M$. Indeed $\phi^i = 0$ if and only if $\phi^i\otimes \operatorname{I} = 0$ where $\operatorname{I}: M\to M$ denotes the identity morphism. From the vantage point of linear algebra, the strictly upper triangular and strictly lower triangular matrices are the first examples of nilpotent elements, and the construction of nilpotent bundle morphisms begins in the same way. One may, in low genus, exploit the wealth of known information about holomorphic bundles and their maps --- for example, the Birkhoff-Grothendieck splitting when working in genus $0$ --- to create globally strictly triangular morphisms as needed.

In service to stating our main results, let us fix the definitions of the very stable and wobbly loci, recognizing that very stable bundles are semistable for a twist of sufficiently large degree. This follows after the assumption of property (\ref{ss}) in Lemma \ref{vst}. Below, $\mathcal{M}_X^{ss}(r, d)$ denotes the moduli space of semistable rank $r$, degree $d$ bundles on $X$.

\begin{definition}
The subsets of $\mathcal{M}_X^{ss}(r, d)$ consisting, respectively, of $L$-very stable and $L$-wobbly bundles are, respectively, the \emph{very stable locus} and the \emph{wobbly locus}. We denote these by the symbols $\mathcal{V}(r, d, L)$ and $\mathcal{W}(r, d, L)$, respectively.
\end{definition}

We now state the main theorem that will be developed in later sections.

\begin{theorem*}
Let $X$ be a complex elliptic curve, let $r$ be an integer greater than or equal to $2$, let $d$ be any integer, and let $\mathcal{E}(r, d)$ denote the set of isomorphism classes of indecomposable bundles of rank $r$ and degree $d$ on $X$.  \begin{enumerate}
\item Any holomorphic bundle $E\in \mathcal{E}(r, 0)$ is canonically wobbly. A stable bundle $E\in\mathcal{E}(r, d)$, should one exist, is canonically very stable. A polystable bundle $E \cong \bigoplus_{i = 1}^n E_i$, where the summands have equal slopes, is canonically very stable if and only if $E_i\ncong E_j$ for all $i\neq j$. 
\item Let $L\to X$ be a twist of degree $1$. A bundle $E\in \mathcal{E}(2, 1)$ is $L$-very stable if and only if $\det(E)\neq L$.
\item Let $L\to X$ be a twist with a degree $\geq 2$. A stable bundle $E\in\mathcal{E}(r, d)$ is $L$-wobbly.
\item If $(r, d) = h > 1$, then the canonically wobbly locus forms a closed irreducible subvariety of codimension $1$ inside $\mathcal{M}_X^{ss}(r, d)$, and hence is a divisor.
\item The first Chern class of the wobbly divisor is $2(h\cdot \eta - \sigma)$ for some generators $\eta, \sigma\in H^2(\emph{Sym}^h(X), \mathbb{Z})$.
\end{enumerate}
\end{theorem*}

\noindent\textbf{Acknowledgements.} The authors thank Christian Pauly for preliminary discussions on this subject after a lecture of his during the Workshop on the Geometry, Algebra, and Physics of Higgs Bundles (23w5082). The workshop, which was co-organized by the second-named author, took place at the Kelowna site of the Banff International Research Station (BIRS) in May-June 2023. Both authors thank the BIRS Kelowna team for providing a hospitable environment for mathematical collaboration as well as the other organizers --- Lara Anderson, Antoine Bourget, and Laura Schaposnik --- for their scientific and logistical efforts.  The second-named author acknowledges Eloise Hamilton for a continuing discourse around twisted Higgs bundles in low genus that was advanced in a significant way at the same venue. The second-named author was partially supported by a Natural Sciences and Engineering Research Council of Canada (NSERC) Discovery Grant during this work. The first-named author was supported by a University of Saskatchewan Graduate Teaching Fellowship (GTF).

\section{Very stable and wobbly bundles}

\subsection{Twisted global nilpotent cone on $\mathbb{P}^1$}
As a warm-up exercise we first focus on $\mathbb{P}^1$. If $t \geq 0$ there does not exist any $t$-very stable bundle over $\mathbb{P}^1$. It is enough to check the statement for $r = 2$. Let $E\cong \mathcal{O}(d_1)\oplus \mathcal{O}(d_2)$. Then either $d_1 - d_2\geq 0$ or $d_2 - d_1 \geq 0$. Furthermore, without loss of any generality, we can take $d_1 - d_2 + t\geq 0$. Define 
\begin{equation}
\phi = \begin{bmatrix}
    0 & \phi'\\
    0 & 0
\end{bmatrix}.
\end{equation} Here, $\phi'$ is a nonzero section of a line bundle, appearing as a global component of $\phi$. This $\phi$ is a non-zero nilpotent Higgs field. 
\begin{proposition}
    Let $t < 0$; then, a bundle $E \cong \mathcal{O}(d_1)\oplus...\oplus \mathcal{O}(d_r)$ is $t$-very stable if and only if $$|d_i - d_j| < -t~~\text{for all}~ i, j.$$
\end{proposition}
\begin{proof}
Supposing $E$ satisfies the inequality from the statement, we necessarily have$$d_i - d_j \geq -|d_i - d_j| > t.$$That is, $d_i - d_j + t < 0$ for all $i, j$. As a consequence, $\phi\in H^0(\mbox{End}E\otimes \mathcal{O}(t))$ is zero and very stability of $E$ is confirmed. Let us consider the other direction now. In case there exists $i\neq j$ so that $0 > t > d_j - d_i$, then $d_i - d_j + t > 0$. We choose a section $s\in H^0(\mathcal{O}(d_1 - d_2 + t))$. Construct $\phi$ by inserting $s$ at the $(i, j)$-th (or the $(j, i)$-th) entry of a matrix, based on an ordered choice of a global basis; furthermore, place $0$ everywhere else so that $\phi$ is nilpotent. We also remark that the differences of any two Grothendieck numbers of a $t$-very stable bundle $E$ are bounded and a semistable bundle is $t$-very stable.
\end{proof}

\subsection{Stable versus very stable bundles}The following results are well known, but we outline their proofs.
\begin{lemma}\label{exact}
Let $E, E', E''$ be vector bundles on a curve $X$ satisfying an exact sequence $$0\to E'\to E\to E''\to 0$$ and $L$ be a line bundle on $X$ such that $H^0(\operatorname{Hom}(E'', E'\otimes L))\neq 0$. Then $E$ is not very stable. Conversely, if $E$ is an $L$-wobbly bundle of rank $2$ with a nonzero nilpotent element $\phi\in H^0(\mbox{End}E\otimes L)$ then there exists a line subbundle $E'$ such that $H^0(\operatorname{Hom}(E/E', E'\otimes L))\neq 0$.
\end{lemma}
\begin{proof} Let $\psi$ denote a nonzero bundle morphism $E''\to E'\otimes L$. According to the following diagram we can choose $\phi = i\otimes 1\circ \psi\circ q$ which satisfies $\phi^2 = 0$.
\begin{equation}
\begin{tikzcd}
0 \arrow[r] & E'\arrow[r, "i"] & E\arrow[r, "q"] & E''\arrow[r]\arrow[dl, "\psi"] & 0\\
& 0 \arrow[r] & E'\otimes L \arrow[r, "i\otimes 1"] & E\otimes L \arrow[r] & E''\otimes L \arrow[r] & 0
\end{tikzcd}
\end{equation} 
Assume that $\phi:E\to E\otimes L$ such that $\phi\neq 0$ and $\phi^2 = 0$. Then we choose $E' = \ker(\phi)$ considering $\phi$ as a homomorphism of locally free sheaves. In a similar diagram as above taking $E'' = E/E' = \operatorname{Im}(\phi)$ we obtain a diagram where $\psi$ will be a well-defined bundle morphism $\operatorname{Im}(\phi)\to \ker(\phi)\otimes L$ due to $\phi^2 = 0$ and exactness of the following sequence of bundles\\ 
$$\begin{tikzcd}
0 \arrow[r] & \ker(\phi)\otimes L \arrow[r, "i\otimes 1"] & E\otimes L \arrow[r] & \operatorname{Im}(\phi)\otimes L \arrow[r] & 0.\end{tikzcd}$$ \end{proof}
\begin{lemma}\label{vst}
Let $X$ be a smooth complex curve.
\begin{enumerate}
\item If $E$ is semistable and $\deg(L) < 0$ then $E$ is $L$-very stable. \item\label{ss} Let $X$ be a curve with genus $g_X\geq 1$ and $L$ be a line bundle with $\deg(L)\geq 2(g_X - 1)$. An $L$-very stable bundle $E$ is slope-semistable. In particular, $g_X > 1$ implies $E$ is slope-stable.
\end{enumerate}
\end{lemma}
\begin{proof} 
To prove the first statement we claim that $H^0(\mbox{End}E\otimes L) = 0$. Choose, if possible, a nonzero $\phi\in H^0(\mbox{End}E\otimes L)$. As a morphism of sheaves $\phi$ can not be injective because we must respect the obvious inequality $\deg(E) > \deg(E\otimes L)$. So we apply the semistability condition on the subbundle $\ker(\phi)\subset E$ and on the saturated subbundle of $\mbox{Im}(\phi)$ contained in $E\otimes L$, to arrive at the contradiction that $\deg(L) \geq 0$.\\
To prove the latter statement we choose a slope-destabilizing subbundle $F$ of $E$, that is, $\mu(F) > \mu(E)$. We claim that $H^0(\operatorname{Hom}(E/F, F\otimes L)) \neq 0$ under this assumption. Observe that 
\begin{equation}
\mu(F) - \mu(E) > 0 \geq \left(1 - \frac{\mbox{rank}(F)}{\mbox{rank}(E)}\right)(g_X - 1 -\deg(L)). 
\end{equation}
Appealing to Riemann-Roch formula we obtain 
\begin{align*}
& h^0(\operatorname{Hom}(E/F, F\otimes L))\\
& = h^1(\operatorname{Hom}(E/F, F\otimes L)) + \mbox{rank}(F)\cdot (\mbox{rank}(E) - \mbox{rank}(F))(1 - g_X)\\
& + \deg((E/F)^*\otimes F\otimes L) > 0.
\end{align*}
According to the proof we see that under the assumption $g_X > 1$ we can replace slope-semistability with slope-stability. Here we have the following inequality 
\begin{equation}
\mu(F) - \mu(E) \geq 0 > \left(1 - \frac{\mbox{rank}(F)}{\mbox{rank}(E)}\right)(g_X - 1 -\deg(L)). 
\end{equation}
Even if $g_X = 1$ and $\deg(L) > 0$ we obtain stability of an $L$-very stable bundle.
\end{proof}

\subsection{Additivity and multiplicativity of vector bundles on an elliptic curve}
The crux of Atiyah's work is the exploration of indecomposable bundles and their multiplicative structures on elliptic curves. Tu organizes necessary information about the moduli of stable and semistable bundles out of Atiyah's construction of indecomposable bundles. Moreover, they computed the respective cohomology groups of such moduli space of bundles. We collect a list of useful results (with some changes in language) from both papers (\cite{Ati} and \cite{Tu}) in service to the readers.\\ 

A point $A$ on an elliptic curve $X$ gives rise to a unique line bundle $\mathcal{O}(A)\to X$ so that every nonzero section of $\mathcal{O}(A)$ vanishes at $A$. This is an isomorphism of complex manifolds (between the elliptic curve and its Picard variety of degree $1$), usually called the \textit{Abel-Jacobi map}. There is an isomorphism $\mbox{Pic}^1X\cong \mbox{Pic}^0X$ by $L\mapsto L\otimes \mathcal{O}(A)^{-1}$. Commonly the elements of a symmetric product of a curve are well defined divisors of the degree same as the power of the symmetric product. We will use the description of $\mbox{Sym}^h(X)$ as the collection of effective divisors interchangeably with the classical definition. More generally, there is a holomorphic Abel-Jacobi map $\alpha:\mbox{Sym}^h(X)\to \mbox{Pic}^hX$ assigning a divisor of $h$ points on $X$ to its corresponding line bundle of degree $h$. The tangent and the cotangent (canonical) bundles on $X$ are isomorphic to the trivial line bundle $X\times \mathbb{C}$.
\begin{enumerate}\label{list}
\item (\textit{The Uniqueness Theorem}) Let $E\in \mathcal{E}(r, d)$ with $d\geq 0$. Then (i) $\dim H^0(E) = h^0(E) = d$ when $d > 0$ and $0$ or $1$ when $d = 0$, (ii) if $d < r$, there is a trivial subbundle $I_s = \bigoplus_{i=1}^s\mathcal{O}_X$ of $E$ while $E' = E/I_s$ is indecomposable with $h^0(E') = s$.

\item (\textit{Existence Theorem}) Let $E'\in \mathcal{E}(r', d)$ with $d\geq 0$ and if $d = 0$ let $h^0(E') \neq 0$. Then there exists a bundle $E\in \mathcal{E}(r, d)$ unique up to isomorphism, given by an extension $$0\to I_s\to E\to E'\to 0$$ where $r = r' + s$ and $s = d$ when $d > 0$ and $1$ when $d = 0$.

\item (\textit{Extensions of indecomposable bundles of degree $0$}) There exists a vector bundle $F_r\in \mathcal{E}(r, 0)$ with $h^0(F_r)\neq 0$ that fits into an exact sequence $$0\to \mathcal{O}_X\to F_r\to F_{r-1}\to 0$$ and a bundle $E\in \mathcal{E}(r, 0)$ is of the form $E \cong F_r\otimes L$ for some $r$-torsion element $L\in \mbox{Pic}^0(X)$. The bundle $F_r$ is self dual i.e $F_r\cong F_r^*$ and $h^0(F_r\otimes L)\neq 0$ if and only if $h^0(L)\neq 0$.

\item (\textit{Indecomposable bundles with nonzero degrees}) Let $r, d$ be mutually prime integers. Then $E\in\mathcal{E}(r,d)$ is uniquely identified with its determinant $\det(E)\in \mbox{Pic}^d(X)$. Suppose that $A\in\mbox{Pic}^1(X)\cong X$ is a chosen base point. There is a unique bundle $E_A(r, d)\in \mathcal{E}(r,d)$ so that we are able to identify $E\in \mathcal{E}(r,d)$ as $E_A(r, d)\otimes L$. To be precise, $E_A(r, d)$ is the unique indecomposable bundle of rank $r$ and degree $d$ that admits determinant $\mathcal{O}(A)^d$. Moreover, $$E_A(r, d)^*\cong E_A(r, -d).$$ Even if $r$ and $d$ admit the greatest common divisor $h$, a bundle $E_A(r, d)$ is uniquely defined by a bijective correspondence between $\mathcal{E}(h, 0)$ and $\mathcal{E}(r, d)$. Related results are generalized with step by step reductions to the `mutually prime' cases.

\item (\textit{Multiplicative structures}) If $r, s > 1$ the bundle $F_r\otimes F_s$ decomposes as a finite direct sum of bundles $F_i$'s. If $r, d$ are mutually prime integers then $$E_A(r, d)\otimes F_h\cong E_A(rh, dh)$$ is indecomposable for any integer $h$. For chosen mutually prime pairs of integers $r, d$ and $r', d'$ so that $r, r'$ are mutually prime we have $$E_A(r, d) \otimes E_A(r', d') \cong E_A(rr', rd' + r'd).$$

\item (\textit{Semistability of indecomposable bundles}) For a chosen base point $A$ in $\mbox{Pic}^1X$, the canonical indecomposable bundle $E_A(r, d)$ is semistable for any $r, d$. If $(r, d) = h > 1$ and $r = r'h$ and $d = d'h$ then $E_A(r, d)$ is strictly semistable from the following exact sequence of bundles $$0\to E_A(r', d')\to E_A(r, d)\to E_A(r - r', d - d') \to 0.$$ This confirms stability of $E$ precisely if $r$ and $d$ are mutually prime. In particular, $F_r$ is semistable for all $r\in\mathbb{N}$.

\item (\textit{Decomposable semistable bundles}) Set $h = (r, d)$ and $r = r'h;~ d = d'h$. Every semistable bundle of rank $r$ and degree $d$ on $X$ is strongly equivalent to a bundle of the form $$E_A(r', d')\otimes\bigoplus_{i = 1}^h M_i,$$ where $M_i\in \mbox{Pic}^0X$ with $M_i^{n'} = \mathcal{O}_X$. There is an isomorphism $E_A(r', d')\otimes\oplus_{i = 1}^h M_i\mapsto \oplus_{i = 1}^h (M_i^{n'}\otimes A)$. 

\item (\textit{Semistable bundles and symmetric products of an elliptic curve}) Let the moduli space of (grading equivalent) semistable and the moduli space of (isomorphic) stable bundles of rank $r$ and degree $d$ on an elliptic curve $X$ be denoted with $\mathcal{M}_X^{ss}(r, d)$ and $\mathcal{M}_X^s(r, d)$ respectively. Suppose that $h$ is the greatest common divisor of $r$ and $d$. Then $$\mathcal{M}_X^{ss}(r, d)\cong \mbox{Sym}^h(X)$$ and $$\mathcal{M}_X^{s}(r, d)\cong \begin{cases}
    \emptyset;~ h > 1\\
    X;~ h = 1
\end{cases}.$$ 

\item\label{comm} There is a commutative diagram 
\begin{equation}
\begin{tikzcd}
\mathcal{M}_X^{ss}(r, d) \arrow[r, "\cong"] \arrow[d, "\det"]
& \mbox{Sym}^h(X) \arrow[d, "\alpha "] \\
\mbox{Pic}^dX \arrow[r, "\otimes \mathcal{O}(A)^{h - d}"]
& \mbox{Pic}^hX
\end{tikzcd}
\end{equation}
In short we explain the commutative diagram as following. Choose an element $E$ in $\mbox{Sym}^h(X)$ and represent it in two ways, first as a bundle $E_1\oplus...\oplus E_h$ each with the same rank and the same degree and take its determinant $L$. Then consider the same element $E$ as a divisor $p_1 +...+ p_h$ and consider the line bundle $\otimes_{i = 1}^h\mathcal{O}(p_i)$. These two line bundles differ by a tensor product of a power of $\mathcal{O}(A)$ so that $\det(E_i) = \mathcal{O}(A)^{\mbox{rank}(E_i)}$.
\item $\mbox{Pic}(\mathcal{M}_X^{ss}(r, d))\cong \mbox{Pic}(\mbox{Pic}(X))\oplus \mathbb{Z}$. Denote $\mathcal{M}_X^{ss}(r, L)$ the semistable bundles of a fixed determinant $L$. Then $\mathcal{M}_X^{ss}(r, L)\cong \mathbb{P}^{h - 1}$. We derive it from the fact that a fiber of the Abel-Jacobi map is the projectivization of the linear system of effective divisors of $L\otimes\mathcal{O}(A)^{h - d}$. 
\end{enumerate}

\section{Very stable and wobbly loci}
\subsection{Characterization of very stable and wobbly bundles}
We first consider $0\leq d < r$ and manage any other degree through adjusting via the division algorithm. We handle the indecomposable bundles at first then the case of the decomposable ones. 

\begin{remark}
We should be attentive to specific marginal cases. On an elliptic curve a trivially (or canonically) twisted very stable bundle is semistable but it does not necessarily define a point in the moduli space of semistable bundles and yields no information surrounding the very stable locus. Hence we should work out the case of decomposable (polystable) bundles to establish a few topological results of the very stable locus (or the wobbly locus) while handling separately the case of the indecomposable bundles. The same comment makes sense for any twist of degree $0$.
\end{remark}
\begin{enumerate}
\item Let $E\in \mathcal{E}(r, 0)$. For a twist $\deg(L) < 0$ we have $E$ is $L$-very stable due to the fact that $E$ is semistable. However a more elementary argument is available at our exposure, utilizing the fact that there exists an element $M\in \mbox{Pic}^0(X)$ so that $E\cong F_r\otimes M$. Decompose 
\begin{equation}
F_r\otimes F_r = \bigoplus\limits_{i = 1}^n F_{k_i} 
\end{equation}
in to indecomposable components and observe that $H^0(F_i\otimes L) = 0$ since $F_r\otimes L$ is indecomposable with a negative degree (cf. \cite{banerjee2023generalized} Lemma 3.19) we arrive at $H^0(\mbox{End}(E)\otimes L) = 0$. Hence $E$ is $L$-very stable.\qed\\
Next we choose a twist $L$ as $\deg(L) > 0$ or $L = \mathcal{O}_X$ we apply Lemma \ref{exact} on the exact sequence 
\begin{equation}
0\to M\to E\to F_{r-1}\otimes M\to 0 
\end{equation}
whereas $H^0(\mbox{Hom}(F_{r-1}\otimes M, M\otimes L)) = H^0(F_{r-1}\otimes L)\neq 0$. We conclude that $E$ is $L$-wobbly.\qed

\item Now we consider $0 < d < r$. For $\deg(L) < 0$ we have $E\in \mathcal{E}(r, d)$ is $L$-very stable once again as $E$ is semistable (Lemma \ref{exact}).\qed\\
Choose a twist $L$ with $\deg(L) > 0$. We focus on a line bundle $L$ so that $\deg(L)\geq 2$. We derive that $E$ is $L$-wobbly. Choose an exact sequence 
\begin{equation}
0\to \mathcal{O}_X\to E\to E'\to 0. 
\end{equation}
We remark that $E'$ may decompose in to bundles of smaller ranks. The inequality $0 < d < r + (r - 2)$ implies $\deg(L) \geq 2 > \frac{d}{r-1}$. We obtain 
\begin{equation}
\deg(E'^*\otimes L) = (r - 1)\deg(L) - d > 0. 
\end{equation}
By $\ref{exact}$ we have $E$ is $L$-wobbly via Riemann-Roch.\qed\\
Let $\deg(L) = 1$. We choose the rank and the degree of $E$ so that $r > d + 1$. The exactly same reasoning $E$ is $L$-wobbly. We sharpen our reasoning to investigate the $L$-very stable pairs in the case of $r = d + 1$. At first we focus at $d = 1$. Suppose that $E$ is a rank $2$ bundle of degree $1$ with determinant $\delta\in \mbox{Pic}^1(X)$. Due to its stability each line subbundle $\xi\subset E$ admits a degree $\leq 0$. In particular, there is an exact sequence of bundles 
\begin{equation}
0\to \xi\to E\to \xi^{-1}\otimes\delta \to 0.
\end{equation}
If $E$ admits a line subbundle $\xi$ such that $\xi^2\otimes L = \delta$ then $E$ is $L$-wobbly. If such a line subbundle does not exist then $E$ is $L$-very stable due to \ref{exact}. In general we use the exact sequence 
\begin{equation}
0\to I_d\to E\to E'\to 0 
\end{equation}
where $E'$ is a degree $1$ line bundle isomorphic to the determinant of $E$. If $E' \cong L$ we can conclude $E$ is $L$-wobbly which confirms that the $L$-wobbly locus in $\mathcal{E}(d + 1, d)$ is nonempty. We frame a relevant conjecture \ref{con} in this context. We further choose the twist $L = \mathcal{O}_X = T^*X$. If $(r, d) = 1$ then $E\in \mathcal{E}(r, d)$ is stable, thus simple and obviously very stable. However, if $(r, d) > 1$ a bundle $E$ is not necessarily stable and $H^0(\mbox{End}E)$ may or may not admit a nonzero nilpotent Higgs field. We handle such bundles up to a grading equivalent polystable bundle in Remark \ref{dec}.\qed
\begin{conjecture}\label{con}
Consider the holomorphic one-to-one correspondence $$\det:\mathcal{E}(d+1, d)\cong X\to X\cong \operatorname{Pic}^1(X).$$ The set of $L$-very stable locus is an open dense subset $\det^{-1}(\operatorname{Pic}^1(X)\backslash\{L\}) = X\backslash\det^{-1}(L)$ for $d\geq 2$.
\end{conjecture}
\end{enumerate}

\begin{remark}
we are yet to settle very stability of bundles for a nontrivial twist $L\in \mbox{Pic}^0X$. For $r\geq 2$ we focus on two particular cases: either $L$ is an $r$-torsion element in $\mbox{Pic}^0(X)$ with order $r$ or not an $r$-torsion element. Suppose that $L\in \mbox{Pic}^0X$ admits order $r$ and $E$ is a bundle of rank $r$ and $\phi\in H^0(\mbox{End}E\otimes L)$ admitting a nonzero determinant $\det(\phi)\in \mathbb{C}$. It is easy to derive that $$\phi:E\to E\otimes L$$ is an isomorphism of bundles. Assuming moreover that $E$ is stable we conclude that $E$ is $L$-very stable. On the other hand $E$ is very stable if and only if $H^0(\mbox{End}E\otimes L)$ is trivial, given that $\det(\phi) = 0$ for any $\phi$ (since $\phi\in H^0(\mbox{End}E\otimes L)$ is nilpotent).\\

Choose a twist $L$ so that its order is greater than $r$. This is an extension of the previous case. Once again, a bundle $E$ of rank $r$ is very stable if and only if $H^0(\mbox{End}E\otimes L)$ is trivial. For example, $F_r\in \mathcal{E}(r, 0)$ is $L$-very stable.\qed
\end{remark}

\begin{remark}\label{dec}
We can briefly talk about decomposable $L$-very stable bundles on an elliptic curve $X$. To respect the stability property, we discard all the twists $L$ with positive degrees (Lemma \ref{vst}). Assume $\deg(L) \leq 0$ and $E \cong \bigoplus_{i = 1}^n E_i$ decomposes in to indecomposable summands. If $E$ is $L$-very stable then $$\deg(E_i^*\otimes E_j\otimes L) \leq 0$$ precisely if $$|\mu(E_i) - \mu(E_j)|\leq -\deg(L)$$ holds for all $i, j$. In particular, if $\deg(L) = 0$ we moreover impose that $E_i$'s are stable, each with slope $\mu$. Then the bundle $E$ is $L$-wobbly if and only if for some $i\neq j$, $$E_i\cong E_j\otimes L.$$ This assertion is true due to the following well known fact: Let $E_1, E_2$ be stable bundles of same slope. Then $H^0(\mbox{Hom}(E_1, E_2)) \neq 0$ if and only if $E_1\cong E_2$. For example, $\bigoplus_{i = 1}^n E$ is $\mathcal{O}_X$-wobbly for any stable bundle $E\to X$. On this note we conclude the characterization of the twisted very stable and wobbly bundles on an elliptic curve $X$.\qed
\end{remark}

We summarize the whole discussion in form of the following theorem.
\begin{theorem}\label{maint}
Let $X$ be a complex elliptic curve.
\begin{enumerate}
\item Any bundle $E\in \mathcal{E}(r, 0)$ is canonically wobbly. A stable bundle $E\in\mathcal{E}(r, d)$, if exists, is canonically very stable. A polystable bundle $E \cong \bigoplus_{i = 1}^n E_i$ (summands of equal slopes) is canonically very stable if and only if $E_i\ncong E_j$ for all $i\neq j$. 
\item Let $L\to X$ be a twist of degree $1$. A bundle $E\in \mathcal{E}(2, 1)$ is $L$-very stable if and only if $\det(E)\neq L$.
\item Let $L\to X$ be a twist with a degree $\geq 2$. A stable bundle $E\in\mathcal{E}(r, d)$ is $L$-wobbly.
\end{enumerate}
\end{theorem}
\begin{corollary} Let $X$ be a complex elliptic curve.
\begin{enumerate}
\item A semistable canonically wobbly bundle is not stable. 
\item Let $L\to X$ be a twist of degree $1$. The wobbly locus $\mathcal{W}(2, 1, L)$ is a point on $X$.
\item Let $L\to X$ be a twist of degree $\geq 2$. Then $\mathcal{W}(r, d, L)\cong \mathcal{M}_X^{ss}(r, d)$.
\end{enumerate}
\end{corollary}
\begin{remark}
Let $E$ be an indecomposable bundle of degree $0$. Its grading is a direct sum of line bundles which are isomorphic to each other, thus the grading bundle is wobbly too.
\end{remark}
\subsection{Topology of the canonically very stable locus and the wobbly locus on curves}
We dedicate this subsection to a rapid recollection of the results concerning the canonically wobbly loci on curves of higher genera. We furthermore discuss the merits and demerits of these results in case of an elliptic curve. The geometry of the canonically very stable and the canonically wobbly locus of bundles was formally introduced by Laumon \cite{Gera} on the algebraic stack of rank $n$ bundles $\mbox{Fib}_{X, n}$, on a complex algebraic curve $X$. Laumon proved that the nilpotent cone in the cotangent bundle $T^*\mbox{Fib}_{X, n}$ is the support of a closed reduced Lagrangian submanifold of $T^*\mbox{Fib}_{X, n}$ of dimension $n^2(g_X - 1)$ and later concluded that very stable bundles form an open dense subset in $\mbox{Fib}_{X, n}$. The very stable bundles live inside the open dense subset of slope-semistable bundles for $g_X > 0$ (cf. \cite{Gera} Proposition (3.5)) by Lemma \ref{exact} and Lemma \ref{vst}. In the same article \cite{Gera} Laumon referred to a result of Drinfeld stating that the collection of the wobbly bundles form a pure closed subset in $\mbox{Fib}_{X, n}$ of codimension $1$. Recently mathematicians (\cite{Pal}, \cite{Palpau}, \cite{Peon}) investigated in to topological properties of the moduli space of the very stable and the wobbly bundles within the moduli space of Higgs bundles on curves $X$ of genus $\geq 2$. At this stage, we briefly recollect features of the moduli space of semistable $L$-twisted pairs for a chosen line bundle $L$.\newline\newline
For the rest of the article, we restrict our focus to $L = K_X$ or $\deg(L) > 2g_X - 2$ on a curve $X$ of genus at least $1$. Nitsure (\cite{Nitin}) proved existence of a quasi-projective separated noetherian scheme $\mathcal{M}_X(r, d, L)$ of finite type over $\mathbb{C}$ parametrizing the collection of strongly equivalent semistable $L$-twisted pairs which contains an open subscheme $\mathcal{M}_X'(r, d, L)$ parametrizing the collection of isomorphism classes of stable pairs. In particular, there is a smooth open subscheme $\mathcal{M}_0'$ of complex dimension $\dim_{\mathbb{C}} \mathcal{M}_0' = r^2\deg(L) + 1 + h^1(X, L)$ parametrizing the stable pairs which admits a stable underlying bundle. This scheme $\mathcal{M}_0'$ is integral. In particular, $\mathcal{M}_0'(K_X)$ is the cotangent bundle $T^*\mathcal{M}_X^{s}(r, d)$ that admits hyperk\"ahler structures. For a general twist the space of stable pairs $\mathcal{M}_0'$ admits a K\"ahler structure. There is a proper morphism $H: \mathcal{M}_X(r, d, L)\to \oplus_{i = 1}^r H^0(X, L^i)$ evaluating the characteristic coefficients of pairs. In \cite{Peon}, Pauly and Pe\'on-Nieto proved a chain of important criteria as a device to locate the very stable bundles inside $T^*\mathcal{M}_X^{s}(r, d)$ on a curve of genus $\geq 2$. The following result is a version of \cite{Peon} Theorem 1.1 with minimal changes in the language and the notation.

\begin{theorem}\label{verys}
Let $E$ be a stable bundle of rank $r$ and degree $d$ and $V_E$ be the complex vector space $H^0(\mbox{End}E\otimes K_X)$. Then the following statements are equivalent.
\begin{enumerate}
\item $E$ is very stable.
\item $V_E$ is closed in $\mathcal{M}_X(r, d, K_X)$.
\item The restriction of the Hitchin morphism $H$ on $V_E$ is proper.
\item $H_{V_E}$ is quasi-finite.
\end{enumerate}
\end{theorem}
The proof of the above statement engages deeply with the topology of the limit points coming from the usual $\mathbb{C}^*$-action on the moduli space $\mathcal{M}_X(r, d, K_X)$. We recall, as a standard fact, that this moduli space is semiprojective i.e. the limit $\lim_{\lambda\to 0} (E, \lambda\cdot\phi)$ which exists in the projective completion of the moduli space lies inside the moduli space itself. To prove this specific assertion, they choose the path of showing a contradiction. They modify a rational map from a complex surface to $\mathcal{M}_X(r, d, K_X)$ by resolving its indeterminacy locus and construct a morphism from a connected union of projective lines to the Zariski closure of $V$. This morphism assumes values $(E, 0)$ and $(F, \psi)\in \bar{V}\backslash V$ at two distinct points and leads to the conclusion that $F = E$ to respect semistability of $F$. A favorable estimate of the dimensions of affine spaces plays an important role in this argument. Recall that a proper morphism between two affine schemes of the same dimension, in particular, the Hitchin base $\mathcal{B}$ and the vector space $H^0(\mbox{End}E\otimes K)$ of the same complex dimension $1 + r^2(g_X - 1)$, is quasi-finite. We can interpret the criterion in a slightly different way. Recall that there is a holomorphic projection morphism $\pi:\mathcal{M}_0'\to \mathcal{M}_X^s(r, d)$ and $\pi^{-1}([E])$ is the vector space $V$ which is closed in $\mathcal{M}_0'$ but may not be closed inside $\mathcal{M}_X(r, d, L)$. Theorem \ref{verys} confirms that on a curve of genus $> 1$, a fiber $\pi^{-1}([E])$ is closed if and only if $E$ is canonically very stable.\\

The nilpotent cone $\mathcal{N}$ is the collection of semistable nilpotent Higgs bundles. It is formally defined as the fiber $H^{-1}(0)$. Due to properness $\mathcal{N}$ is a compact subset (possibly containing singular points) and it is not easy to explore the generally twisted nilpotent cone. For a stable bundle $E\to X$ we have $[(E, 0)]\in H^{-1}(0)$, thus the stable bundles contribute to the nilpotent cone a dimensional estimate at least $r^2(g_X - 1) + 1$. Geometrically, a stable point $(E, 0)$ lies at the bottom of the nilpoent cone and appears as a global minima of the Morse function of stable pairs $(E, \phi)\mapsto 2i\int_X\mbox{trace}(\phi\phi^*)$. To study the geometry of the nilpotent cone, Mathematicians (\cite{Hausel}, \cite{Palpau}) restricted their investigation to the determinant morphism on trace-free canonically twisted pairs of rank $2$ with a fixed determinant $\Lambda$. The main advantage here is that the very stable locus lives inside the `zero-determinant locus' of the stable Higgs bundles as an open subset and simplifies the analysis. The definition that is due to \cite{Palpau}. It futher leads to the fact that the wobbly locus is a reducible divisor.

\begin{definition}\label{def2}
Let $g_X\geq 2$ and $\Lambda$ be a fixed line bundle on the curve $X$. We denote the moduli space of the stable traceless Higgs bundles of rank $2$ with a fixed determinant $\Lambda$ by $\mathcal{M}_X'(2, \Lambda, K_X)$. The determinant morphism $h:\mathcal{M}_X'(2, \Lambda, K_X) \to H^0(K^2)$ is proper and surjective and we mention $h^{-1}(0)$ as the nilpotent cone $\mathcal{N}_h$. The nilpotent cone decomposes as $\mathcal{M}_X^s(2, \Lambda)\cup \mathcal{N}_0$. We identify $\mathcal{M}_X^s(2, \Lambda)$ with the collection of pairs $[(E, 0)]$ and $\mathcal{N}_0$ consists of all nonzero nilpotent stable pairs $[(E, \phi)]$. We observe that $\mathcal{M}_X^s(2, \Lambda)$ is an open subset by the dimensional estimates (\cite{Hausel}) and the image of $\mathcal{N}_0$ under the forgetful rational map $\mathcal{N}_h\dashrightarrow \mathcal{M}_X^s(2, \Lambda)$ is the \textit{wobbly locus}.
\end{definition}
Works of Narasimhan and Ramanan \cite{NR1}, \cite{NR2} laid the foundation of research of stable bundles on complex curves, particularly in the base case of rank $2$. In case of $g_X = 2$, $\mathcal{M}_X^s(2, \delta)$ is isomophic to the intersection of a smooth pencil of quadrics in $\mathbb{CP}^5$. As a known fact we mention that any non-trivial extension of the following type (this exact sequence is mentioned in Lemma \ref{exact}) 
\begin{equation}
0\to \xi^{-1}\to E\to \xi\otimes\delta\to 0
\end{equation}
where $\xi$ is a line bundle of degree $0$, is stable and any two such extensions are isomorphic if and only if these are scalar multiples of each other. Moreover, there is a linear embedding $\mathbb{P}(H^1(\xi^{-2}\otimes \delta^{-1}))\to \mathcal{M}_X^s(2, \delta)$. Pal further restricted the wobbly locus inside the cotangent bundle $T^*\mathcal{M}_X^s(2, \delta)$ and proved (cf. \cite{Pal} Theorem 1.3). that the space of wobbly vector bundles of rank $2$ with determinant $\delta$ is isomorphic to a surface of degree $32$. Recently Pal has completed his proof for an arbitrary rank of the fact that wobbly locus forms a closed subvariety of codimension $1$ inside $\mathcal{M}_X^s(r, \delta)$ generalizing his techniques of extensions of the very stable bundles from the rank $2$ case to an arbitrary rank (cf. \cite{PalG} Theorem 1.1). Pal proved that
a stable bundle $E\in \mathcal{M}_X^s(r, \delta)$ is wobbly if and only if it passes through a nonfree minimal rational curve.\\

The work of Pal and Pauly (\cite{Palpau}) is significant for their explicit computation of the Chern classes of the irreducible subdivisors of the wobbly divisor. Assuming the definition \ref{def2}, we present a theorem of Pal and Pauly (cf. \cite{Palpau} Theorem 1.1) with minimal changes in the language and the notation. We denote $\lambda =\deg(\Lambda)$. Unfortunately, we do not find any major information about the topology of the irreducible components of the wobbly locus of an arbitrary rank on a curve of genus $\geq 2$.
\begin{theorem}
The wobbly locus $\mathcal{W}\subset \mathcal{M}_X^s(2, \Lambda)$ is of pure dimension $1$ and we have the following decomposition for $\lambda = 0$ and $\lambda = 1$, $$\mathcal{W} = \mathcal{W}_{\lceil\frac{g_X - \lambda}{2}\rceil} \cup...\cup \mathcal{W}_{g_X - \lambda}.$$ In particular, all loci $\mathcal{W}_k$ appearing in the above decomposition are divisors. They are all irreducible, except $\mathcal{W}_{g_X}$ for $\lambda = 0$, which is the union of $2^{2g_X}$ irreducible divisors.
\end{theorem}

We explain the divisors $\mathcal{W}_k$ briefly. A sublocus is defined $$\mathcal{W}_k^0 := \{E\in \mathcal{W}: L\subset E;~ \deg(L) = 1 - k~ \text{with}~ H^0(K_XL^2\Lambda^{-1})\neq 0\}$$ and denote by $\mathcal{W}_k$ the Zariski closure of $\mathcal{W}_k^0\subset \mathcal{M}_X^s(2, \Lambda)$. It is deduced that $\mathcal{W} = \cup_{k = 1}^g \mathcal{W}_k$ for $\lambda = 0$ and $\mathcal{W} = \cup_{k = 1}^{g - 1} \mathcal{W}_k$ for $\lambda = 1$. There is a filtration of $\mathcal{W}_k$'s as per \cite{Palpau} Proposition 2.3 that $\bigcup_{k= 1}^{\lceil{\frac{g_X - \lambda}{2}}\rceil -1}\mathcal{W}_k\subset \mathcal{W}_{\frac{g_X - \lambda}{2}}$. Further each of the divisors $\mathcal{W}_k$ represents an element in the Picard group of $\mathcal{M}_X^s(2, \Lambda)$. According to a result by Drezet and Narasimhan we are aware that the Picard group of $\mathcal{M}_X^s(2, \Lambda)$ is isomorphic to $\mathbb{Z}$. The following theorem is Theorem 1.3 \cite{Palpau} which mentions the representatives in the Picard group in terms of the first Chern classes.

\begin{theorem}
    We have the following equality for $\lambda = 0$ and $\lambda = 1$ $$\emph{cl}(\mathcal{W}_k) = 2^{2k}{g_X \choose 2g_X - 2k -\lambda}~ \text{for}~ \lceil\frac{g_X - \lambda}{2}\rceil\leq k\leq g_X - \lambda.$$
\end{theorem}
The proof is technical. We mention a few important steps. The actual strategy follows from decomposing $\mathcal{N}_0$ as a union of images of $\mathcal{W}_k$'s under the aforementioned forgetful rational map. A set of subvarieties $Z_k\subset \mbox{Pic}^{1 - k}(X)$ for $1\leq k \leq g_X - \lambda$ appear as the pre-images of the Brill-Noether Loci $W_{2g_X - 2k - \lambda}$ under the holomorphic map $L(\in\mbox{Pic}^{1 - k}(X)) \mapsto K_XL^2\Lambda^{-1}(\in\mbox{Pic}^{2g_X - 2k - \lambda}(X))$. The fundamental class $[Z_k]$ of $Z_k$ as a subvariety of $\mbox{Pic}^{1 - k}(X)$ is computed as $\frac{2^{4k - 2g_X + 2}}{(2k - g_X + 1)!}\cdot \Theta_{\mbox{Pic}^{1 - k}(X)}^{2k - g_X + 1}$ (\cite{Palpau} Lemma 4.1) wherein we use $\Theta_{\mbox{Pic}^{1 - k}(X)}$ to denote a theta divisor of $\mbox{Pic}^{1 - k}(X)$. Under a classifying map $f:S\to \mathcal{M}_X^s(2, \Lambda)$ the wobbly subloci $\mathcal{W}_k$ is identified as $f^{-1}(\mathcal{W}_k) = \pi_S(\Delta_k\cup (S \times Z_k))$ where $\pi_S: S\times X\to S$ is the projection map. Here $\Delta_k$ is a subvariety of $S\times \mbox{Pic}^{1 - k}(X)$ (cf. page 9 \cite{Palpau}). The fundamental class $[\Delta_k]$ of $\Delta_k$ is computed as $\frac{2^{2g_X - 2k - 2}}{(2g_X - 2k - 1)!}\Theta_S\otimes \Theta_{\mbox{Pic}^{1 - k}(X)}^{2g_X - 2k -1}$ given that $\Theta_S$ is the first Chern class of the pullback of an ample generator of $\mbox{Pic}(\mathcal{M}_X^s(2, \Lambda))$ under $f$. Finally, the fundamental class of $f^{-1}(\mathcal{W}_k)$ is the tensor product of $[Z_k]$ with $[\Delta_k]$.

\subsection{Topology of the wobbly locus on an elliptic curve}

The work of Franco G\'omez \cite{Franco} initiates the study of the moduli space of $G$-Higgs bundles on an elliptic curve for a complex reductive group $G = \mbox{GL}(r, \mathbb{C}), \mbox{SL}(r, \mathbb{C}), \mbox{PGL}(r, \mathbb{C}), \mbox{Sp}(2m, \mathbb{C}), \mbox{O}(r, \mathbb{C}), \mbox{SO}(r, \mathbb{C})$.  From this, we collect information on $\mbox{GL}(r, \mathbb{C})$-Higgs bundles. For the group $\mbox{GL}(r, \mathbb{C})$, the results are extensions of Tu's results (\cite{Tu}). We combine \cite{Franco} Proposition 4.2.1 and Proposition 4.2.3 to conclude that a Higgs bundle $(E, \phi)$ is semistable if and only if $E$ is semistable and stable if and only if $E$ is stable (this fact further generalizes for $G$-Higgs bundles). Destabilizing subbundles are obtained out of the Harder-Narasimhan series of a pair $(E, \phi)$ on the elliptic curve. As a corollary (\cite{Franco} Corollary 4.2.4), we have the following statement. Let $(E, \phi)$ be polystable of rank $r$ and degree $d$ and put $h = (r, d)$. Then $$(E, \phi) = \bigoplus\limits_{i = 1}^h (E_i, \phi_i)$$ in which $E_i$ is a stable bundle of rank $r' = r/h$, $d' = d/h$, and $\phi_i = \lambda_i\cdot \emph{I}_{E_i}$. However, even if $E$ is a polystable bundle, a Higgs bundle $(E, \phi)$ is not necessarily polystable. A polystable bundle $E$ on an elliptic curve is grading equivalent to itself but it may not be isomorphic to the underlying bundle of the polystable Higgs bundle that appears as the grading of a pair $(E, \phi)$. This characterization of polystable Higgs bundles sums up to the following result as per \cite{Franco} Theorem 4.3.7 and Proposition 4.3.9. 

\begin{theorem}\label{natp}
There exists a coarse moduli space of strongly equivalence classes of semistable Higgs bundles $\mathcal{M}(\emph{GL}(r, \mathbb{C}))$ isomorphic to $\emph{Sym}^h(\mathcal{O}_X)$ where $h = (r, d)$. There is a natural projection map $$\pi: \emph{Sym}^h(\mathcal{O}_X)\to \emph{Sym}^h(X)$$ generalizing the bundle projection map $\mathcal{O}_X\to X$.
\end{theorem}

We know that the Hitchin morphism is proper and we need to update the image of the Hitchin morphism inside the Hitchin base. The process establishes a Higgs bundle analogue of the commutative nature between the Abel-Jacobi morphism and the determinant morphism for bundles. Here the Hitchin base is $\mathcal{B}_{r, d} = \oplus_{i = 1}^r H^0(\mathcal{O}_X) = \mathbb{C}^r$. Remember that each summand $(E_i, \phi_i)$ of a polystable Higgs bundle $(E, \phi)$ is identified as $(E_i, \lambda_i\mbox{I}_{E_i})$ for some complex numbers $\lambda_1,..., \lambda_h$ (because a stable bundle is simple). We can represent $\phi$ as a long diagonal matrix $$(\underbrace{\lambda_1,..., \lambda_1}_{r'~\text{times}},..., \underbrace{\lambda_h,..., \lambda_h}_{r'~\text{times}})$$ factoring $r = r'\cdot h$. The characteristic coefficients of such a diagonal matrix are symmetric polynomials of the diagonal entries. Assembling these information we conclude that the image of the Hitchin morphism $B(r, d) =\mbox{Sym}^h(\mathbb{C})$. A canonical embedding of $B(r, d)\hookrightarrow \mathcal{B}_{r, d}$ is given as $(v_1,..., v_h)\mapsto (\underbrace{v_1,..., v_1}_{r'~\text{times}},..., \underbrace{v_h,..., v_h}_{r'~\text{times}})$. The following commutative diagram summarizes the discussion;
\begin{equation}
\begin{tikzcd}
\mathcal{M}_X^{ss}(r, d, \mathcal{O}_X) \arrow[r, "\cong"] \arrow[d, "\text{H}"]
& \mbox{Sym}^h(\mathcal{O}_X) \arrow[d, "\alpha "] \\
B(r, d) \arrow[r, "\pi^{(h)}"]
& \mbox{Sym}^h(\mathbb{C})
\end{tikzcd}
\end{equation}
while $\pi^{(h)}$ the canonical projection morphism $\pi^{(h)}: \mbox{Sym}^h(\mathcal{O}_X)\to \mbox{Sym}^h(\mathbb{C})$. Spectral curves lie at another corner of this theory. Recall that the spectral curves in this specific case are not well behaved because
they are always reduced and reducible subschemes of $\mathbb{P}(\mathcal{O}_X \oplus \mathcal{O}_X)$. The fiber of $H$ of points of the form $(\lambda_1,..., \lambda_h)$ so that $\lambda_i\neq\lambda_j$ is isomorphic to an abelian variety $X^h$. It is an elementary observation that a polystable Higgs bundle $\bigoplus_{i = 1}^h (E_i, \phi_i)$ is nilpotent if and only if $\lambda_i$'s are all $0$. We conclude that the nilpotent cone is isomorphic to $\mbox{Sym}^h(X)$ and neither a very stable nor a wobbly bundle can be specifically located inside the moduli space $\mathcal{M}_X^{ss}(r, d)$ simply looking at its fiber under $\pi$. A fiber $\pi^{-1}([\oplus_{i = 1}^h E_i]) = V$ admits a vector space structure. Fix $E_1,..., E_h$ and there are well defined operations $$[\oplus_{i = 1}^h (E_i, \phi_i)] + [\oplus_{i = 1}^h(E_i, \phi_i')] = [\oplus_{i = 1}^h(E_i, \phi_i + \phi_i')]$$ and $$\lambda\cdot[\oplus_{i = 1}^h (E_i, \phi_i)] = [\oplus_{i = 1}^h (E_i, \lambda\cdot\phi_i)].$$ This is a complex vector space of dimension $h$. Each fiber is closed in the moduli space of semistable pairs. Thus it is not a practical idea to locate very stable bundles with its fiber over $\pi$ (\ref{natp}). We will further see that the wobbly locus $\mathcal{W}(r, d, \mathcal{O}_X)$ is closed. As a consequence $\pi^{-1}(\mathcal{W}(r, d, \mathcal{O}_X))$ is closed. Here we warn the reader that we shift to the definition of the very stable locus and the wobbly locus given in section \ref{state} for the rest of the article. \\

We can afford partial comments on the $L$-wobbly loci on an elliptic curve for $\deg(L) > 0$. Note that the `quasi-finiteness' criterion of $L$-very stability is inconclusive for $\deg(L) > 0$. Denote $V = H^0(\mbox{End}E\otimes L)$, then 
\begin{equation}
\dim_{\mathbb{C}}V = r^2\deg(L) > \frac{r(r+1)}{2}\deg(L) = \sum\limits_{i = 1}^r h^0(X, L^i). 
\end{equation}
As a consequence, we are unable to conclude properness of the restriction of the Hitchin morphism on $V$. On the other hand, for a twist of a low degree, we can partly comment on nilpotent cone over the decomposable locus of rank $2$ bundles.

\begin{example}
Let $L = \mathcal{O}(p)$ be a twist on an elliptic curve $X$ and $E = L_1\oplus L_2$ be a bundle of degree $-1$. We assume that $\deg(L_1)\geq \deg(L_2)$. An $L$-twisted (semi)stable nilpotent element $(E, \phi)$ is written $$\phi = \begin{bmatrix}
    a & b\\
    c & -a
\end{bmatrix}.$$ Here $a, b, c$ are global sections of $L,~ L_2^*L_1 L,~ L_1^* L_2 L$ respectively. To respect stability, we need $c\neq 0$ and $\deg(L_1) = \deg(L_2) + \deg(L)$ that is $L_1\cong L_2\otimes L$. Thus $\phi$ is obtained by scaling $$\phi_0 = \begin{bmatrix}
    A & -A^2\\
    1 & -A
\end{bmatrix}$$ for a fixed global section of $A$ of $L$. It suffices to consider $\phi_0$ with $\begin{bmatrix}
    0 & 0\\
    1 & 0
\end{bmatrix}$ to count distinct elements under equivalence. In total, the collection of such pairs $(E, \phi)$ is represented with orbits of $(1, 1)$-stable chains $$\left(L_2\otimes(L\oplus\mathcal{O}), \begin{bmatrix}
   0 & 0\\
   1 & 0 
\end{bmatrix}\right)$$ under the usual $\mathbb{C}^*$-action. The noncompact complex analytic space $\mbox{Pic}^{-1}(X)\times\mathbb{C}^*$ parametrizes these orbits.\\

Now we handle the case of an even degree. Let $\deg(E) = 0$ and a stable nilpotent $L$-pair $(E, \phi)$. The assumption of stability implies that $\deg(L_1) = \deg(L_2) = 0$. We write $\phi = \begin{bmatrix}
    a & b\\
    c & -a
\end{bmatrix}$, for sections $a, b, c$ of $L, L_2^*L_1L, L_1^*L_2L$. In a marginal case $a = 0$ we have $bc = 0$. But $b\neq 0\neq c$ to entertain stability and this is contradictory. So, $a\neq 0$ and observe from its nilpotency that $b, c$ vanish at the point $p$. This is possible only if $L_1\cong L_2$. So we write $E = L\otimes (L_1\oplus L_1)$. Fix a section $s$ of $L$ so that $\phi$ can be written as $$\phi = s\cdot\phi_0 = s\cdot \begin{bmatrix}
    a' & b'\\
    c' & -a'
\end{bmatrix}$$ where $a', b', c'$ are complex numbers satisfying $a'^2 + b'\cdot c' = 0$. As $c' \neq 0$ we further rescale $\phi_0$ and identify with a representative $$\begin{bmatrix}
    0 & 0\\
    1 & 0
\end{bmatrix}.$$ This is contradictory as $b'\neq 0$. Thus there are no stable nilpotent pair for $\deg(E) = 0$. However, if we entertain semistability of nilpotent pairs, only the trivial pairs survive and contribute a smooth compact connected subset isomorphic to $\emph{Sym}^2(\mbox{Pic}^0(X))$.\qed
\end{example}

Now we focus on the topology and holomorphic structures of the wobbly loci on an elliptic curve. The wobbly bundles for higher twisted are already sorted out but there is still an open end to explore for trivially twisted wobbly locus. To investigate the topology of such (canonically) wobbly bundles it suffices to feature only the polystable bundles. To construct a very stable bundle by hand we choose a finite sequence of distinct points in the Picard group $\{p_1,... p_{h}\}$ (after fixing a base point $A$) which define a unique bundle $E = E_A(r', d')\otimes ({p_1}\oplus {p_2}\oplus...\oplus {p_h})$. We use a combinatorial point of view that yields a connection between the closed subvarieties of the symmetric powers of curves with the wobbly loci.\\

Let $k, i_1,..., i_k$ be integers $\geq 1$. The numbers $i_1,..., i_k$ are said to be the \textit{weights} with \textit{weight count} $k$ and \textit{total weight} $h$ if $i_1 +...+ i_k = h$. A grading of a semistable bundle (non uniquely) appears as a \textit{weighted direct sum} 
\begin{equation}
E\cong \underbrace{E_1\oplus...\oplus E_1}_{i_1-\text{times}}\oplus...\oplus\underbrace{E_k\oplus...\oplus E_k}_{i_k-\text{times}} 
\end{equation}
(each summnand is stable of equal rank and equal degree) and note that $E_i$ and $E_j$ are not necessarily unique. We mention this bundle as $i_1 E_1\oplus...\oplus i_k E_k$ alternatively. A generic polystable bundle associate to the weight count $h$ and all weights $1$. Let $T$ denote the finite set of tuples
\begin{equation}
T = \{(k, i_1,..., i_k): i_1 +...+ i_k = h~ \text{and}~ i_1,..., i_k\geq 1;~k\in \mathbb{N}\}.   
\end{equation}

\begin{theorem}\label{com}
Let $(r, d) = h > 1$. Then $\mathcal{W}(r, d, \mathcal{O}_X)$ is a proper closed subset of $\emph{Sym}^h(X)$.
\end{theorem}
\begin{proof}
Let us identify $X$ with the collection of stable bundles of mutually prime rank $r/h$ and degree $d/h$. Choosing an element $(k, i_1,..., i_k)\in T$, define the map $$f_{(k, i_1,..., i_k)}: X^k\to \mbox{Sym}^h(X)$$ by $$E_1\oplus...\oplus E_k\mapsto [i_1 E_1\oplus...\oplus i_k E_k].$$ This is a smooth morphism and we denote its image by $\mathcal{W}(k, i_1,..., i_k)$. This is a compact path connected subset (so closed because $\mbox{Sym}^h(X)$ is compact) in $\mbox{Sym}^h(X)$. We name the subset $\mathcal{W}(k, i_1,..., i_k)$ a \textit{wobbly sublocus} on $X$. Observe that if $(k, i_1,..., i_k)$ and $(k, i_1',..., i_k')$ are tuples of weights in $T$ with the same weight count differing by action of a permutation $\sigma\in S_k$ then $$\mathcal{W}(k, i_1,..., i_k) = \mathcal{W}(k, i_1',..., i_k').$$ We observe that $\mathcal{W}(h, 1,..., 1) = \mbox{Sym}^h(X)$ and in particular the wobbly locus is obtained as the proper compact connected subset in $\emph{Sym}^h(X)$ given as $$\mathcal{W}(h - 1, 2, 1..., 1) = \mathcal{W}(h - 1, 1, 2..., 1) = \mathcal{W}(h - 1, 1, 1..., 2).$$ Finally, it is projective too because $\mbox{Sym}^h(X)$ is projective.
\end{proof}
We sharpen our description of the wobbly locus to compute the global cohomological invariants of the wobbly locus.
\begin{remark}
On an elliptic curve $X$ we have the following statements.
\begin{enumerate}
\item Let $i_1,..., i_k$ be weights all $> 1$. Then there is a filtration of wobbly subloci
\begin{align*}
\emptyset\subsetneq\mathcal{W}(k, i_1,..., i_k)\subsetneq \mathcal{W}(k + 1, 1, i_1 - 1,..., i_k)\subsetneq...\subsetneq \mathcal{W}(2k, \underbrace{1,..., 1}_{k-\text{times}}, i_1 - 1,..., i_k -1)\\\subsetneq...\subsetneq \mbox{Sym}^h(X).
\end{align*}
\item Observe that $$\emptyset\subsetneq\mathcal{W}(1, h)\subsetneq \mathcal{W}(2, h - 1, 1) \subsetneq \mathcal{W}(3, h - 2, 1, 1)\subsetneq...\subsetneq \mathcal{W}(h - 1, 2, 1..., 1).$$ We call each of these wobbly subloci a \textit{standard wobbly sublocus} on $X$.
\item A wobbly sublocus $\mathcal{W}(k, i_1,..., i_k)$ is contained inside one of $h - 1$ standard subloci $$\emptyset\subsetneq\mathcal{W}(1, h)\subsetneq \mathcal{W}(2, h - 1, 1) \subsetneq \mathcal{W}(3, h - 2, 1, 1)\subsetneq...\subsetneq \mathcal{W}(h - 1, 2, 1..., 1).$$
\end{enumerate}
\end{remark}
The wobbly locus of strictly semistable bundles is an irreducible divisor of $\mbox{Sym}^h(X)$. In fact, there is a more useful definition of $f_{(k, i_1,..., i_k)}$ in algebraic geometry assuming the points of the symmetric powers as effective divisors. Let us consider $N = (n_1,..., n_k)$ and $$h = \sum\limits_{l = 1}^k n_l\cdot i_l.$$ There is an orientation preserving \textit{diagonal morphism} (cf. \cite{Mac}) by the following weighted direct sum 
\begin{equation}
    \Phi_N: \mbox{Sym}^{n_1}(X)\times...\times \mbox{Sym}^{n_k}(X) \to \mbox{Sym}^h(X); D_1 +...+ D_k\mapsto i_1D_1 +...+ i_kD_k.
\end{equation}

In a particular case $n_1 =...= n_k = 1$ we obtain the sublocus $\mathcal{W}(k, i_1,..., i_k)$ as the image of $\Phi_N$. Moreover, if $i_1 >...> i_k$ then $\Phi_N$ is an isomorphism of complex manifolds 
\begin{equation}
\mbox{Sym}^{n_1}(X)\times...\times \mbox{Sym}^{n_k}(X)\cong\mathcal{W}(k, i_1,..., i_k).
\end{equation}
\begin{remark}
The above definition of a diagonal morphism makes sense for curves of an arbitrary genus. In case $g_X = 0$, MacDonald (\cite{Mac}) identified the image spaces of $\Phi_N$ (for specific tuples $N$) in $\mbox{Sym}^h(\mathbb{P}^1) = \mathbb{P}^h$ with known projective curves and projective surfaces. However, we are not aware of any better description of the image of the diagonal morphism inside the projective variety $\mbox{Sym}^h(X)$ for curves of nonzero genera.
\end{remark}

Let $n = n_1 +...+ n_k$. Computation of the image of the fundamental class --- that is, the image of the generator $1$ of $$H_{2n}(\mbox{Sym}^{n_1}(X)\times...\times \mbox{Sym}^{n_k}(X))\cong \mathbb{Z}$$ under the diagonal morphism $\Phi_N$ is a standard problem in intersection theory. The standard wobbly subloci $\mathcal{W}(1, h),..., \mathcal{W}(h - 1, 2, 1,..., 1)$ are proper closed submanifolds of $\mbox{Sym}^h(X)$ isomorphic to $X \times \mbox{Sym}^{h - s}(X)$ for $s = h, h - 1,..., 2$ respectively. Indeed, observe that $s\cdot 1 + 1\cdot (h - s) = h$ and there are maps $$\pi(s, 1^{h - s}): X\times \mbox{Sym}^{h - s}(X)\to\mbox{Sym}^h(X)$$ given as 
\begin{equation}
(p, [D])\mapsto [s\cdot p + D]
\end{equation}
so that these subloci are images of $\pi(s, 1^{h - s})$. In \cite{Mac}, this variety appears and is denoted with $\Delta(s, 1^{h - s})$. Each of these varieties is irreducible because $X\times \mbox{Sym}^{h - s}$ is irreducible. Also observe that the usual structure sheaf of complex holomorphic functions $X\times \mbox{Sym}^{h - s}$ leaves it reduced. Thus $X\times \mbox{Sym}^{h - s}$ is smooth and integral and its admits complex dimension dimension $h - s$. As a proper closed subvariety $\mathcal{W}(r, d)$ admits dimension $\leq h - 1$. By induction, a sublocus $\mathcal{W}(k, h - k + 1, 1,..., 1)$ admits Krull dimension $k$ (otherwise $\mbox{Sym}^h(X)$ has dimension $> h$ from the above filtration of closed subvarieties). In particular, the aforementioned filtration of standard subloci supports the fact that Krull dimension of the wobbly locus inside $\mbox{Sym}^h(X)$ is $h - 1$. Finally, every standard wobbly sublocus is a projective variety because $X$ is a projective variety and a symmetric product of $X$ is a projective variety. We organize the discussion in the following commutative diagram. This concludes a version of Drinfeld's claim (\cite{Gera}) on an elliptic curve.
\begin{equation}
\begin{tikzcd}\label{comm3}
\emptyset \arrow[r, phantom, sloped, "\subsetneq"] & X \arrow[r, phantom, sloped, "\subsetneq"] \arrow[d, phantom, sloped, "\cong"] & \hdots \arrow[r, phantom, sloped, "\subsetneq"] & X\times\mbox{Sym}^{h -2}(X) \arrow[r, phantom, sloped, "\subsetneq"] \arrow[d, phantom, sloped, "\cong"] & \mbox{Sym}^h(X) \arrow[d, phantom, sloped, "\cong"] \\
\emptyset \arrow[r, phantom, sloped, "\subsetneq"] & \mathcal{W}(1, h) \arrow[r, phantom, sloped, "\subsetneq"] & \hdots \arrow[r, phantom, sloped, "\subsetneq"] & \mathcal{W}(h - 1, 2, 1..., 1) \arrow[r, phantom, sloped, "\subsetneq"] & \mathcal{M}_X^{ss}(r, d).
\end{tikzcd}
\end{equation}
For $p\in X$, we have a natural closed embedding $i_p: \mbox{Sym}^m(X)\hookrightarrow \mbox{Sym}^{m + 1}(X)$ with $D\mapsto D + p$ The image of $i_p$ is the \textit{hyperplane} of a fixed component $p$ inside $\mbox{Sym}^{m + 1}(X)$. Going one step further, $X\times\mbox{Sym}^{h - s}(X)\hookrightarrow X\times\mbox{Sym}^{h - s + 1}(X)$ is given by $(p, D)\mapsto (p, i_p(D))$. The embedding $[s\cdot p + D] = [(s - 1)\cdot p + i_p(D)]$ (as $p\in X$ and $D\in\mbox{Sym}^{h - s}(X)$) fits in the commutative diagram \ref{comm3}. 
\begin{corollary}
The very stable locus $\mathcal{V}(r, d, \mathcal{O}_X)$ is an open dense subset of $\mathcal{M}_X^{ss}(r, d)$. We denote $$\mathcal{V}(k, i_1,..., i_k) = \mathcal{M}_X^{ss}(r, d)\backslash \mathcal{W}(k, i_1,..., i_k).$$ There is a filtration of open dense subsets in $\emph{Sym}^h(X)$ by $$\emph{Sym}^h(X)\supsetneq\mathcal{V}(1, h)\supsetneq \mathcal{V}(2, h - 1, 1) \supsetneq \mathcal{V}(3, h - 2, 1, 1)\supsetneq...\supsetneq \mathcal{V}(h - 1, 2, 1..., 1) = \mathcal{V}(r, d, \mathcal{O}_X).$$
\end{corollary}
The following proposition is immediate.
\begin{proposition}
Let $h \geq 2$. For each of the numbers $i = 1,..., h - 1$, there is a $(h - 2)!$-fold holomorphic covering $f_i: X^{h - 1}\to \mathcal{W}(r, d, \mathcal{O}_X)$ restricting $f_{(h - 1, 1,...,\underbrace{2}_{i-\text{th position}},..., 1)}$.
\end{proposition}

\begin{proof}
We address $f_i: X^{h - 1}\to \mathcal{W}(r, d, \mathcal{O}_X)$ defined as 
\begin{equation}
(E_1,..., E_{h - 1})\mapsto [E_1\oplus...\oplus 2 E_i\oplus...\oplus E_{h - 1}].
\end{equation}
These maps are holomorphic surjective defined between compact complex manifolds of the same complex dimension $h - 1$. Let us choose the open dense subset of $\mathcal{W}(r, d, \mathcal{O}_X)$ consisting of the polystable bundles with exactly two repeating components.
\begin{equation}
V = \mathcal{W}(h - 1, 2, 1..., 1)\backslash \mathcal{W}(h - 2, 3, 1..., 1). 
\end{equation}
To prove that the degree of $f_i$ is $(h - 2)!$, it suffices to focus on $i = 1$ only. Choose two different representatives of a polystable bundle (in the image of $V$)
\begin{equation}
2 E_1\oplus E_2\oplus...\oplus E_{h - 1} \cong 2 F_1\oplus F_2\oplus...\oplus F_{h - 1}. 
\end{equation}
From the definition of $f_1$ we have $E_i\ncong E_j$ and $F_i\ncong F_j$ for all $i\neq j$. So $E_1\cong F_1$. Indeed $E_1\cong F_j$ and $E_1\cong F_k$ for the repeating component in the left hand side. This is possible precisely if $j = k = 1$. Thus $E_2,..., E_{h - 1}$ and $F_2,..., F_{h - 1}$ are unique up to a permutation leading to the fact that each fiber of $f_1$ contains exactly $(h - 2)!$ distinct points.
\end{proof}

\subsection{Computation of topological invariants of the wobbly locus}
 An element of the wobbly locus, as an effective divisor, defines a unique holomorphic line bundle on $\mbox{Sym}^h(X)$. Recall that the finite holomorphic covering map $X^h\to \mbox{Sym}^h(X)$ of $h!$ sheets by the underlying group action defines a group homomorphism $\mbox{Pic}(\mbox{Sym}^h(X))\to \mbox{Pic}(X^h)$. Also $L\in \mbox{Pic}(X^h)$ is of the form $L = \otimes_{i = 1}^h\mbox{pr}_i^*L_i$ while $\mbox{pr}_i: X^h\to X$ denotes the $i$-th projection morphism. We can identify the Picard variety of $\mbox{Sym}^h(X)$ as the class of representatives $[\otimes_{i = 1}^h\mbox{pr}_i^*L_i]$ so that 
 \begin{equation}
 \bigotimes\limits_{i = 1}^h\mbox{pr}_i^*L_i \cong \bigotimes\limits_{i = 1}^h\mbox{pr}_{\sigma(i)}^*L_i 
\end{equation}
 holds for all $\sigma\in S_h$. On a smooth generic curve of genus $> 1$ the N\'eron-Severi group of $\mbox{Sym}^h(X)$ is generated with two elements (\cite{Arb}, \cite{Kou}). Instead we recall the computations of the fundamental classes of the wobbly subloci (in the subring of algebraic parts of the graded rational cohomology ring $H^*(\mbox{Sym}^h(X), \mathbb{Q})$) on an elliptic curve from \cite{Mac}.\\

We first fix our conventions. We denote with `$\cdot$' the cup-product operation of integral cohomology groups or simply juxtapose two elements. On a complex elliptic curve $X$, recall that $H^1(X, \mathbb{Z})$ is generated by two elements $\alpha, \alpha'$ over the ring of integers and $\alpha\cdot\alpha' = \beta$ generates $H^2(X, \mathbb{Z})$ induced by the orientation of $X$. The set of elements $\{\alpha,\alpha',\beta\}$ generates the graded cohomology ring $H^*(X, \mathbb{Z})$. The relations satisfied for $\alpha, \alpha', \beta$ are the following: $$\alpha\cdot\alpha' = \alpha'\cdot\alpha = \beta^2 = 0;~ \alpha\cdot\alpha' = -\alpha'\cdot\alpha = \beta.$$ Let $n\in\mathbb{N}$. Then $\alpha_k := 1\otimes...\otimes\underbrace{\alpha}_{k-\text{th component}}\otimes...\otimes 1$ and $\alpha_k' := 1\otimes...\otimes\underbrace{\alpha'}_{k-\text{th component}}\otimes...\otimes 1$ and $\beta_k := 1\otimes...\otimes\underbrace{\beta}_{k-\text{th component}}\otimes...\otimes 1$. Furthermore, $\xi~(\text{resp.}~\xi';~\eta) = \sum_{k = 1}^n \alpha_k~(\text{resp.}~\sum_{k = 1}^n \alpha_k'$ and $\sum_{k = 1}^n\beta_k)$. The graded cohomology ring $H^*(\mbox{Sym}^n(X), \mathbb{Z})$ is generated by $\xi, \xi', \eta$ over $\mathbb{Z}$ under the conditions that $\xi, \xi'$ anticommute with each other and commute with $\eta$ and $(\xi\cdot\xi' - \eta)\cdot\eta^{n - 1} = 0$ (cf. \cite{Mac} 6.3). We denote with $\sigma$ the element $\xi\cdot\xi'$.\\

As usual we identify $\mbox{Pic}^0(X)$ with $X$ and $p\in X$ is a fixed base point. Then there is a holomorphic map $\phi(n): \mbox{Sym}^n(X)\to X$ given by $D\mapsto \mathcal{O}(D - np)$ and $\phi(n)^{-1}(u)$ is a $PGL(n - 1, \mathbb{C})$ bundle which has fiber $\mathbb{P}^{n - 1}$ at $u\in X$. For $n\geq 2$ there is a vector bundle $E(n)\to X$ of rank $n$ so that $\mbox{Sym}^n(X)$ is the projective bundle associated to $E(n)$. The map $i_p:\mbox{Sym}^{n - 1}(X)\to\mbox{Sym}^n(X)$ defines a projective subbundle $\mathbb{P}(E(n - 1))\hookrightarrow \mathbb{P}(E(n))$. The degree of $E(n)$ is given by $1 - \alpha\cdot\alpha'$. Recall that the minimal equation satisfied by $\eta$ is $\eta^n + \phi(n)^*(\deg(E(n))\cdot \eta^{n - 1} = 0$ (it is only a special case of the general formula available for a projective bundle). Comparing this equation with $(\xi\cdot\xi' - \eta)\cdot\eta^{n - 1} = 0$ the degree of $E(n)$ is computed. The morphism $i_p$ further induces ${i_p}_*: H^i(\mbox{Sym}^{n - 1}, \mathbb{Z})\to H^{i + 2}(\mbox{Sym}^n(X), \mathbb{Z})$ with ${i_p}_*(\omega) = \eta\cdot\omega$. The Chern class of $\mbox{Sym}^n(X)$ is computed as $(1 + \eta)^{n - 1}\cdot(1 + \eta - \sigma)$.\\

Recall the diagonal morphism once again \begin{equation}
    \Phi_N: \mbox{Sym}^{n_1}(X)\times...\times \mbox{Sym}^{n_k}(X) \to \mbox{Sym}^h(X); D_1 +...+ D_k\mapsto i_1D_1 +...+ i_kD_k.
\end{equation} The cohomology class of the image of $\Phi_N$ is the coefficient of $y_1^{n_1}...y_k^{n_k}$ in the polynomial $$P^{\nu - 1}\eta^{n - \nu - 1}(P\eta + Q(\eta - \sigma))$$ while $P = i_1y_1 +...+ i_ky_k$,~$Q = (i_1^2 - i_1)y_1 +...+ (i_k^2 - i_k)y_k$ and $\nu = n_1 +...+ n_k$. In particular, we write $\delta_s = \mbox{cl}(\Delta(s, 1^{h - s}))$ by the formula $$\delta_s = s\cdot[h\cdot\eta^{s - 1} - (s - 1)\eta^{s - 2}\cdot\sigma]$$ and further $s = 2$ we have the cohomology class of the wobbly locus which gives the first chern class of the wobbly divisor. This statement can be viewed as an application of Poincar\'e duality combining the push-pull projection formula on forms $\omega = \xi^a\xi'^b\sigma^c\eta^q\in H^{2\nu}(\mbox{Sym}^h(X), \mathbb{Z})$ where $a, b, c\in\{0, 1\}$ and $q = 2\nu - a - b - c$ (while $\nu$ runs from $1$ to $h$).

\begin{remark}
Let $L$ be a line bundle on $X$ and $\deg(L) = d\geq 2$. Here we modify $L$ with $L_1 = L\otimes\mathcal{O}(A)^{h - d}$ about the base point $A$ for which we identify $\mathcal{M}_X^{ss}(r, d) \cong \mbox{Sym}^h(X)$. We recall that the collection $\mathcal{M}_X^{ss}(r, L)$ of semistable bundles with fixed determinant $L$ is identified with $\mathbb{P}(H^0(L_1))$ inside the symmetric power $\mbox{Sym}^h(X)$ (fiber of the Abel-Jacobi morphism). We will investigate the wobbly loci $\mathcal{W}(1, h, L),..., \mathcal{W}(h - 1, 2, 1..., 1, L)$ of fixed determinant $L$. $\mathcal{W}(1, h, L)$ is a finite set of points. Rest of these varieties are smooth closed subvarieties isomorphic to the total spaces of projective bundles $\mathbb{P}(E_{L_1}(h - s))$ for $s = h - 1,..., 2$ and $\mathbb{P}(E_{L_1}(h - s))$ admits $\mathbb{P}(E_{L_1}(h - s - 1))$ as a closed subvariety. For a chosen point $p\in X$ and chosen $s$ in the above integer bound, consider the holomorphic line bundle $M(s) = L_1\otimes\mathcal{O}(-s\cdot p)$ of degree $h - s$ on $X$. We first give a reader-friendly description of projective fibers. We identify $\mathbb{P}(H^0(M(s)))$ as $\mathbb{P}^{h - s - 1}$ based at a point $p$. Fix a nonzero section $\mathfrak{s}(p)$ vanishing at $p$. There is an injective morphism of $\mathbb{C}$-vector spaces $\otimes \mathfrak{s}(p): H^0(M(s + 1))\to H^0(M(s))$ (denoting with $\mathfrak{s}(p)$ a nonzero section of $\mathcal{O}(p)$). This leads to a closed embedding $\mathbb{P}(H^0(M(s + 1)))\to \mathbb{P}(H^0(M(s)))$ contributing the hyperplane at infinity.\\

We obtain the standard wobbly loci of a fixed determinant $L$ as a projective bundle on the elliptic curve $X$. We define a bundle $E_{L_1}(h - s)\to X$ with fiber $\{(x, s): s\in H^0(L_1(-s\cdot x))\}$ at $x\in X$. Existence of the bundle $E_{L_1}(h - s)$ is approved by Grauert's semicontinuity theorem (\cite{Hart} Corollary 12.9). Consider the Poincar\'e line bundle (\cite{Birkenhake1992ComplexAV}, Theorem 5.1 or \cite{Mich}) $\mathcal{L}\to\mbox{Pic}^1(X)\times X$ for which $\mathcal{L}_{|\{M\}\times X}\cong M$. Consider the sheaf $$\mathcal{L}' = {\pi_{\mbox{Pic}^1(X)}}_*(\mathcal{L}^{-s}\otimes \pi_X^*L_1)\to \mbox{Pic}^1(X)$$ that is locally free of rank $h - s$ same as the complex dimension of $H^0(L(-s\cdot x))$. Since $X\cong \mbox{Pic}^1(X)$, we obtain a bundle on $X$ associated to $\mathcal{L}'$ which we call $E_{L_1}(h - s)$. Finally, the wobbly loci of fixed determinant $L$ are obtained as $\mathbb{P}(E_{L_1}(h - s))$ that admits complex dimension $h - s$. There is an injective morphism of projective bundles: at $x\in X$ choose $s\in H^0(L_1(h - s - 1))$ defined by a divisor $\sum a_jx_j$ and map it to the effective divisor $\sum a_jx_j + x$. We pass to the action by nonzero complex numbers to write a closed immersion $\mathbb{P}(E_{L_1}(h - s - 1))\hookrightarrow\mathbb{P}(E_{L_1}(h - s))$ and organize this fact in the following filtration of closed subvarieties (a version of Drinfeld's claim for very stable bundles of a fixed determinant line bundle)
\begin{equation}\label{eq2}
\begin{tikzcd}
\emptyset \arrow[r, phantom, sloped, "\subsetneq"] & \mathcal{W}(1, h, L) \arrow[r, phantom, sloped, "\subsetneq"] \arrow[d, phantom, sloped, "\cong"] & \hdots \arrow[r, phantom, sloped, "\subsetneq"] & \mathcal{W}(h - 1, 2, 1..., 1, L) \arrow[r, phantom, sloped, "\subsetneq"] \arrow[d, phantom, sloped, "\cong"] & \mathcal{M}_X^{ss}(r, d, L) \arrow[d, phantom, sloped, "\cong"] \\
\emptyset \arrow[r, phantom, sloped, "\subsetneq"] & \{L(i)\}_{i = 1}^{h^2} \arrow[r, phantom, sloped, "\subsetneq"] & \hdots \arrow[r, phantom, sloped, "\subsetneq"] & \mathbb{P}(E_L(h - 2))\arrow[r, phantom, sloped, "\subsetneq"] & \mathbb{P}^{h - 1}.
\end{tikzcd}
\end{equation}
Here $L(1),..., L(h^2)$ count $h^2$ many distinct line bundles so that $L(i)^h = L_1$. This commutative diagram \ref{eq2} appears as the restriction of \ref{comm3} on the moduli space of semistable bundles of fixed determinant $L$. However, we are not yet able to conclude that $\mathbb{P}(E_{M_1}(n))\cong\mathbb{P}(E_{M_2}(n))$ (at least as smooth projective varieties, even if not as $\mathbb{P}^{n - 1}$-bundles on $X$) in case, $\deg(M_1) = \deg(M_2)$. There are more unsettled questions left. Denote the induced morphism $\pi:\mathbb{P}(E_{L_1}(n))\to X$. Then the graded cohomology ring $H^*(\mathbb{P}(E_{L_1}(n)))$ is the polynomial ring generated by the first Chern class of the tautological line bundle $c_1(\pi^*\mathcal{O}(1)) = \xi$ with coefficients from $H^*(X)$ under the constraint $\xi^n + c_1(E(n))\cdot \xi^{n - 1} = 0$. The restriction of the diagonal morphism $\Phi:\mathbb{P}(E_{L_1}(h - s))\to\mathbb{P}^{h - 1}$ induces a morphism $\Phi_*: H^*(\mathbb{P}(E_{L_1}(h - s)), \mathbb{Z})\to H^*(\mathbb{P}^{h - 1}, \mathbb{Z})$. There is a generator $\alpha\in H^2(\mathbb{P}^{h - 1}, \mathbb{Z})$ of $H^*(\mathbb{P}^{h - 1}, \mathbb{Z})$ so that $\alpha^h = 0$. Unfortunately, at this stage we do not know the exact formula of $\Phi_*(\xi)$ in terms of (the powers of) $\alpha$. 
\end{remark}

The symmetric powers of a curve $X$ being a generalized construction of the complex projective spaces deserve a few comments on cellular decompositions. The $n$-th projective space $\mathbb{P}^n$ is isomorphic to $\mbox{Sym}^n(\mathbb{P}^1)$ with a cell decomposition of $\mathbb{P}^n = \mathbb{C}^n\cup \mathbb{C}^{n - 1}\cup...\cup \mathbb{C}^0$. For higher genus curves, a cell decomposition of a symmetric product is not easily achievable. We provide a \textit{pseudo-cell decomposition} for a symmetric power of an elliptic curve $X$, which is possible because of its group law. First, observe that $\mbox{Sym}^{n + 1}(X)\backslash i_p\left(\mbox{Sym}^{n}(X)\right)$ is the open subset $\mbox{Sym}^{n + 1}(X\backslash\{p\})$ that avoids the divisors containing the point $p$. Also the subset $\mbox{Sym}^{n + 1}(X\backslash\{p\})$, for an arbitrary choice of a base point $p$, is a noncompact connected complex manifold unique up to an isomorphism. Let $q$ be another point on $X$. There is an automorphism on $X$ that maps $p$ to $q$: for any $z\in X$ consider the divisor $z - p + q$ that is linearly equivalent to a unique point $w\in X$. From this automorphism we can identify $\mbox{Sym}^n(X\backslash\{p\})\cong\mbox{Sym}^n(X\backslash\{q\})$ for all $n$. The noncompact curve $X\backslash\{p\}$ is an affine curve. We denote $\mbox{Sym}^n(X\backslash\{p\}) = V_n$ and for practical purposes we can choose a base point $p$ to define the map $i_p$. Note that $\overline{V_n} = \mbox{Sym}^n(X)$, thus $V_n$ is analogous to a \textit{Bruhat cell} and $\mbox{Sym}^n(X)$ is analogous to a \textit{Schubert variety}. It is customary to mention that $\mbox{Sym}^n(X) = V_n\cup V_{n - 1}\cup...\cup V_0$ (based at $p$) followed by a flag of closed subvarieties $$\mbox{Sym}^0(X)\subsetneq \mbox{Sym}^1(X)\subsetneq...\subsetneq \mbox{Sym}^n(X).$$

Let $L\to X$ be a line bundle of degree $n + 1$. The projective space $\mathbb{P}^n = \mathbb{P}(H^0(L))$ is a subvariety of $\mbox{Sym}^{n + 1}(X)$. The set of linearly equivalent effective divisors of $L$ that avoid $p$ (that is, the holomorphic sections of $L$ not vanishing at $p$) is $\mbox{Sym}^{n + 1}(X\backslash\{p\})\cap\mathbb{P}^n$. Observe that a global section of $L$ that vanishes at $0$ is a tensor product of a section of $L\otimes\mathcal{O}(p)^{-1} = L(- p)$ with a nonzero section of $\mathcal{O}(p)$. Thus we obtain, via the restriction of a pseudo-cell decomposition, the usual cell decomposition on projetive spaces below $$\mbox{Sym}^{n + 1}(X\backslash\{p\})\cap\mathbb{P}(H^0(L)) = \mathbb{P}(H^0(L))\backslash \mathbb{P}(H^0(L(-p)))\cong \mathbb{C}^n.$$ The divisor at infinity $\mathbb{P}(H^0(L(-p)))$ admits an very nice $n$-dimensional complex manifold structure with an underlying vector space structure. Let $v_1,..., v_n$ be linearly independent global sections of $L(-p)$. Then the linearly independent sections $v_1\otimes\mathfrak{s}(p),..., v_n\otimes\mathfrak{s}(p)$ of $L$ extend to a global basis of $H^0(L)$ attaching a global section $v_{n + 1}$. Obviously, $v_{n + 1}$ does not vanish at $p$. We consider the $(n + 1)$-tuples of coefficients of the ordered basis $\{v_1\otimes\mathfrak{s}(p),..., v_n\otimes\mathfrak{s}(p), v_{n + 1}\}$ with the nonzero contributions from $v_{n + 1}$.\\ 

We obtain similar straight forward cases of \textit{pseudo-cell decomposition} on the standard wobbly loci.
\begin{equation}
X\times\mbox{Sym}^n(X) = (X\times V_n)\cup...\cup (X\times V_0).
\end{equation}
We mimic the cell decomposition for the wobbly loci of a fixed determinant $L$ by restricting the previous cell decomposition. At a fixed base point $p$ we have a canonical closed embedding $i_p:\mathbb{P}(E_{L_1(-p)}(h - s))\to \mathbb{P}(E_{L_1}(h - s + 1))$ by $(x, D)\mapsto (x, D + p)$. We obtain the open subset $$\mathbb{P}(E_{L_1}(h - s + 1))\backslash i_p(\mathbb{P}(E_{L_1(-p)}(h - s))) = \mathbb{P}(E_{L_1}(h - s + 1)\backslash E_{L_1}(h - s))$$ and 
\begin{equation}
\mathbb{P}(E_{L_1}(h - s + 1)) = \mathbb{P}(E_{L_1}(h - s + 1)\backslash E_{L_1}(h - s))\cup...\cup \mathbb{P}(E_{L_1}(1))
\end{equation}
Here we are yet to conclude if the holomorphic structure of the open set $\mathbb{P}(E_{L_1}(h - s + 1)\backslash E_{L_1}(h - s))$ depends at all on the choice of a base point $p$.
\begin{remark}
As a final exercise, we compute the Poincar\'e polynomial of a standard wobbly sublocus $\mathcal{W}(h - s + 1, s, \underbrace{1,...1}_{(h - s)~\text{times}})$ and of a standard wobbly sublocus $\mathcal{W}(h - s + 1, s, \underbrace{1,...1}_{(h - s)~\text{times}}, L)$ with a fixed determinant $L$.  We do this by applying a K\"unneth decomposition on the Poincar\'e polynomials of $X$, $ \mbox{Sym}^{h - s}$, and projective bundles on a curve. The Betti numbers of symmetric products of a curve are by now classical, and were computed in \cite{Mac}. We obtain
\begin{align*}
&\mathcal{P}(\mathcal{W}(h - s + 1, s, \underbrace{1,...1}_{(h - s)~\text{times}}), z) = (1 + 2z + z^2)\cdot(1 + 2z +...+ 2z^{2(h - s) -1} + z^{2(h - s)})\\
& = z^{2(h - s) + 2} + 4z^{2(h - s) + 1} + 7z^{2(h - s)} + 8z^{2(h - s) - 1} +...+ 8z^3 + 7z^2 + 4z + 1
\end{align*} and 
\begin{align*}
&\mathcal{P}(\mathcal{W}(h - s + 1, s, \underbrace{1,...1}_{(h - s)~\text{times}}, L), z) = \sum\limits_{i = 0}^{2(h - s)} h^i(\mathbb{P}(E_L(h - s)))\cdot z^i\\
& = \begin{cases}
    h^2;~ h = s;\\
    1 + 2\left(z +...+ z^{2(h - s) - 1}\right) + z^{2(h - s)};~ h < s
\end{cases}\end{align*} 
\end{remark}
Here, $s$ runs from $2$ to $h$. We implement the formula $h^i(\mathbb{P}(E_L(h - s)), \mathbb{Z}) = \sum_{t = 0}^{h - s - 1} h^{i - 2t}(X, \mathbb{Z})$.

\printbibliography

\end{document}